\documentclass[12pt]{amsart}
\usepackage{amsmath,amscd,amssymb,a4wide,graphicx}
\usepackage{amsthm}
\usepackage{dsfont}

\usepackage[foot]{amsaddr}

\usepackage{setspace}
\usepackage{xcolor}
\usepackage[colorlinks,citecolor=blue,linkcolor=blue,urlcolor=blue]{hyperref} 

\usepackage[dvipsnames]{pstricks}
\usepackage{pstricks-add}
\usepackage{hyperref}
\usepackage{textcomp}

\usepackage{geometry}
\newcounter{theorem}
\newtheorem{thm}[theorem]{Theorem}
\newcounter{thm}
\newtheorem{conj}[thm]{Conjecture}
\newtheorem{lemma}[thm]{Lemma}
\newtheorem{fakeprop}{Proposition}
\newtheorem{fakethm}{Theorem}

\newtheorem{definition}[thm]{Definition}

\newcounter{pro}
\newtheorem{prop}[pro]{Proposition}

\newtheorem*{Remark}{Remark}

\DeclareMathOperator{\tr}{tr}
\DeclareMathOperator{\Sym}{Sym}
\DeclareMathOperator{\Symp}{Symp}
\DeclareMathOperator{\Hilb}{Hilb}
\DeclareMathOperator{\Ham}{Ham}

\DeclareMathOperator{\Vect}{Span}

\DeclareMathOperator{\id}{id}

\DeclareMathOperator{\diag}{diag}

\DeclareMathOperator{\Hom}{Hom}

\DeclareMathOperator{\End}{End}

\begin{document}

\author{V.V. Fock, A. Thomas}
\title{Higher complex structures}
\address{Universit\'e de Strasbourg, CNRS, IRMA UMR 7501, 67084 Strasbourg, France}
\email{fock@math.unistra.fr} 
\email{thomas@math.unistra.fr}



\begin{abstract}
We introduce and analyze a new geometric structure on topological surfaces generalizing the complex structure. To define this so called higher complex structure we use the punctual Hilbert scheme of the plane. The moduli space of higher complex structures is defined and is shown to be a generalization of the classical Teichm\"uller space. We give arguments for the conjectural isomorphism between the moduli space of higher complex structures and Hitchin's component. 
\end{abstract}

\maketitle

\section*{Introduction}

Poincar\'e's uniformization theorem links complex structures on a surface $\Sigma$ to homomorphisms from the fundamental group of the surface to $PSL(2,\mathbb{R})$, the automorphism group of the hyperbolic plane. With this, Teichm\"uller space $\mathcal{T}_\Sigma$ can be identified with the connected component of the character variety of faithful and discrete representations of the fundamental group in $PSL(2,\mathbb{R})$: $$\mathcal{T}_\Sigma \cong \Hom^{\text{discrete}}(\pi_1(\Sigma), PSL(2,\mathbb{R}))/PSL(2,\mathbb{R}).$$ In his celebrated paper \cite{Hit.1}, Nigel Hitchin proves the existence of a connected component of the character variety for an adjoint group of a split real form of any complex simple Lie group (for instance $PSL(n,\mathbb{R})$) consisting of faithful and discrete representations and parametrized by holomorphic differentials. These components generalize Teichm\"uller space.
His methods are analytic, using the theory of Higgs bundles. Teichm\"uller space also has geometric descriptions: it is the moduli space of hyperbolic structures (metrics with constant curvature -1) and also the moduli space of complex structures. The natural question is then if there is a geometric description of Hitchin's component.


In this paper, we describe and analyze a new geometric structure on surfaces generalizing the complex structure. Conjecturally, the moduli space of the so called higher complex structure is isomorphic to Hitchin's component. This would give a purely geometric approach to higher Teichm\"uller theory. We give arguments in favor of the isomorphism.

The search for a geometric origin of Hitchin's component is of course not new. Goldman, Guichard-Wienhard, Labourie and others describe Hitchin's component via geometric structures on bundles over the surface. For $PSL(3,\mathbb{R})$, this geometric structure is the convex projective structure described by Goldman in \cite{Goldman}. For $n=4$, Guichard and Wienhard describe convex foliated structure on the unit tangent bundle in \cite{Guichard}. Labourie introduces the concept of an Anosov representation in \cite{Lab}. 
The drawback of these constructions is that the bundle on which the geometric structure is defined is not canonically associated to the surface. 

All these generalizations are rigid geometric structures (meaning that the local automorphism group is finite dimensional). Our generalization is not rigid in this sense but behaves as a generalization of complex structures (with local automorphism group holomorphic functions which are infinite dimensional).

We also signal a link to W-algebras since this paper gives partial answers to questions raised in an article of the first author together with Bilal and Kogan \cite{BFK}. This connection will not be treated here.

The paper is structured as follows: first, we review in Section \ref{section2} the approach to complex structures via the Beltrami differential, the most appropriate approach for our generalization. Then in Section \ref{section3}, we introduce the tool we need for generalizing the complex structure: the punctual Hilbert scheme of the plane and its zero-fiber. In the main part \ref{section4}, we define and analyze the higher complex structure. To define the moduli space of higher complex structures, we enlarge the group of diffeomorphisms and look at higher complex structures modulo Hamiltonian diffeomorphisms of the cotangent space $T^*\Sigma$ preserving the zero section setwise. We show that higher complex structure is locally trivializable. The moduli space of higher complex structures is analyzed and compared to Hitchin's component. 
In the last section \ref{section5}, we give arguments in favor for the conjectural isomorphism to Hitchin's component by constructing a spectral curve associated to a higher complex structure. In the appendix \ref{section6} we put details of proofs and computations.

Throughout the paper, $\Sigma$ denotes a connected oriented closed real surface of genus $g \geq 2$. A complex local coordinate on $\Sigma$ is denoted by $z=x+iy$ and its conjugate coordinates on $T^{*\mathbb{C}}\Sigma$ by $p$ and $\bar{p}$. The space of sections of a bundle $B$ is denoted by $\Gamma (B)$.

\section{Complex Structures on Surfaces}
\label{section2}

In this section we review the approach to complex and almost complex structures on a surface via the Beltrami differential. This will prepare the reader to the generalization which will follow in Section \ref{section4}. 

Recall that a \textbf{complex structure} in a manifold is a complex atlas with holomorphic transition functions. Retaining only the information that every tangent space $T_z\Sigma$ has the structure of a complex vector space, i.e. is equipped with an endomorphism $J(z)$ whose square is $-\id$, we obtain an \textbf{almost complex structure}.
A theorem due to Gauss and Korn-Lichtenstein states that every almost complex structure on a surface comes from a complex structure (i.e. can be integrated to a complex structure). Thus, both notions are equivalent on a surface.

So the complex structure is encoded in the operators $J(z)$. These can be better understood by diagonalization. The characteristic polynomial being $X^2+1$, the eigenvalues are $\pm i$. So we need to complexify the tangent space to see the eigendirections: $$T^{\mathbb{C}}\Sigma = T^{1,0}\Sigma \oplus T^{0,1}\Sigma$$ where $T^{1,0}\Sigma$ is the eigenspace associated to eigenvalue $i$.
Furthermore, the eigendirections are conjugated to each other: $T^{1,0}\Sigma = \overline{T^{0,1}\Sigma}$.

Therefore the complex structure is entirely encoded by $T^{0,1}\Sigma$, i.e. a direction in the complexified tangent space. Thus, we can see a complex structure as a section of the projectivized  cotangent bundle $\mathbb{P}(T^{\mathbb{C}}\Sigma)$.

Let us describe this viewpoint in coordinates. To do this, we fix a reference complex coordinate $z=x+iy$ on $\Sigma$. This gives a basis $(\partial, \bar{\partial})$ in $T^{\mathbb{C}}\Sigma$ where $\partial = \frac{1}{2}(\partial_x - i\partial_y) \text{ and } \bar{\partial} = \frac{1}{2}(\partial_x + i\partial_y)$. The generator of the linear subspace $T^{1,0}_z\Sigma$ can be normalized to be  $\bar{\partial}-\mu(z,\bar{z}) \partial$, (where $\mu$ is a coordinate on $\mathbb{C}P^1$, and thus can take infinite value). The coefficient $\mu$ is called the \textbf{Beltrami differential}. There is one condition on $\mu$, coming from the fact that the vector $\bar{\partial}-\mu \partial$ and its conjugate $\partial-\bar{\mu} \bar{\partial}$ have to be linearly independent as eigenvectors corresponding to different eigenvalues. A simple computation shows that this is equivalent to $\mu \bar{\mu} \neq 1$. If we restrict ourselves to complex structures compatible with the orientation of the surface (i.e. homotopy equivalent to the reference complex structure) this condition gives $\mu \bar{\mu} < 1$.

We have only seen the Beltrami differential in a local chart. Changing coordinates $z \mapsto w(z)$ gives $\mu(z,\bar{z}) \mapsto \frac{d\bar{z}/d\bar{w}}{dz/dw}\mu(z,\bar{z})$, so $\mu$ is of type $(-1,1)$, i.e. a section of $K^{-1}\otimes \bar{K}$ where $K=T^{*(1,0)}\Sigma$ the canonical line bundle.

To sum up, we look at a complex structure on a surface as a given linear direction in every complexified tangent space, which is the same as a 1-jet of a curve at the origin. For the generalization, we need to consider rather the \textit{cotangent space} $T^*\Sigma$ (the operators $J(z)$ also act in $T^*_z\Sigma$). Higher complex structures will be given by a $n$-jet of a curve in the complexified cotangent space. To describe this idea precisely in geometric terms, we use the punctual Hilbert scheme of the plane.

\section{Punctual Hilbert scheme of the plane}\label{section3}

We present here the tool necessary for the higher complex structure: the punctual Hilbert scheme of the plane. The \textbf{Hilbert scheme} is the parameter space of all subschemes of an algebraic variety. In general this scheme can be quite complicated but here we are in a very specific case of 0-dimensional subschemes of $\mathbb{C}^2$. Nothing new is presented here, a classical reference is Nakajima's book \cite{Nakajima}.

Consider $n$ points in the plane $\mathbb{C}^2$ as an algebraic variety, i.e. defined by some ideal $I$ in $\mathbb{C}[x,y]$. The function space $\mathbb{C}[x,y]/I$ is $n$-dimensional, since a function on $n$ points is defined by its $n$ values. So the ideal $I$ is of codimension $n$. 
This gives a simple example of a subscheme of dimension zero.
We define the \textbf{length} of a zero-dimensional subscheme to be the dimension of its function space. So the variety of $n$ distinct points is of length $n$. We will see that we get more interesting examples when two or several points collapse into one single point.
The moduli space of zero-dimensional subschemes of length $n$ is called the punctual Hilbert scheme:

\begin{definition}
The \textbf{punctual Hilbert scheme} $\Hilb^n(\mathbb{C}^2)$ of length $n$ of the plane is the set of ideals of $\mathbb{C}\left[x,y\right]$ of codimension $n$: 
$$\Hilb^n(\mathbb{C}^2)=\{I \text{ ideal of } \mathbb{C}\left[x,y\right] \mid \dim(\mathbb{C}\left[x,y\right]/I)=n \}.$$  
\noindent The subspace of $\Hilb^n(\mathbb{C}^2)$ consisting of all ideals supported on 0, i.e. whose associated algebraic variety is $(0,0)$, is called the \textbf{zero-fiber} of the punctual Hilbert scheme and is denoted by $\Hilb^n_0(\mathbb{C}^2)$.
\end{definition}

\begin{Remark}
In the literature the punctual Hilbert scheme of a space $X$ is often denoted by $X^{[n]}$. Here we will deal with several notions of Hilbert schemes, so we prefer the notation $\Hilb^n(X)$.
\end{Remark}

\noindent A theorem of Grothendieck and Fogarty asserts that $\Hilb^n(\mathbb{C}^2)$ is a smooth and irreducible variety of dimension $2n$ (see \cite{Fogarty}). The zero-fiber $\Hilb^n_0(\mathbb{C}^2)$ is an irreducible variety of dimension $n-1$, but it is in general not smooth. 

Let us try to get a feeling for the form of a generic ideal in the Hilbert scheme. Given $n$ generic distinct points in $\mathbb{C}^2$, there is a unique polynomial $Q$ of degree $n-1$ such that these points belong to the graph of $Q$, the Lagrange interpolation polynomial. So we can choose $y=Q(x)$ to be in $I$. If we denote by $x_i$ the $x$-coordinate of the $i$-th point, we get $\prod_i (x-x_i)$ in $I$. These two relations determine already the $n$ points. The ideal $I$ then has two generators and can be put into the form 
$$I=\left\langle x^n+t_1x^{n-1}+\cdots+t_n,-y+\mu_1+\mu_2x+...+\mu_nx^{n-1}\right\rangle.$$

\noindent A point in the zero-fiber of the Hilbert scheme is obtained by collapsing all $n$ points to the origin. A generic point (see \cite{Iarrob}) is obtained when the Lagrange interpolation polynomial admits a limit (for exemple if all points glide along a given curve to the origin). The only condition is that the constant term of $Q$, which is $\mu_1$, has to be zero. Since all $x_i$ become 0, all $t_i$ do as well. So we get an ideal of $\Hilb^n_0(\mathbb{C}^2)$ of the form
$$I=\left\langle x^n,-y+\mu_2x+...+\mu_nx^{n-1}\right\rangle.$$
There are other points in the zero-fiber if $n>2$. For instance in $\Hilb^3_0$ we find the ideal $\left\langle x^2,xy,y^2 \right\rangle$ which is not of the above form (it has three generators).

Notice that for $n=2$, we just get a linear non-vertical direction at the origin, i.e. a point in an affine space. 

There are several equivalent ways to look on the punctual Hilbert scheme. We discuss two of them: as blow-up of a configuration space and as the set of commuting matrices.

When we consider the set of $n$-tuples points as an algebraic variety, we do not see their order. So we see them as a point in the configuration space of (not necessarily distinct) $n$ points, which is the quotient of $(\mathbb{C}^2)^n$ by the symmetric group $\mathcal{S}_n$. This quotient space is singular, since the action of the symmetric group is not free. The theorem of Grothendieck and Fogarty also asserts that $\Hilb^n(\mathbb{C}^2)$ is a minimal resolution of singularities by blow-ups of the configuration space $\Sym^n(\mathbb{C}^2)$. The map from the Hilbert scheme to the configuration space associates to an ideal its support (points in the plane with multiplicity). This is called the \textbf{Chow map}.

To give a link to commuting matrices, we associate to an ideal $I$ of codimension $n$ the linear maps $M_x$ and $M_y$, which are the multiplication operators by $x$ and $y$ respectively in the $n$-dimensional vector space $\mathbb{C}[x,y]/I$. These are two commuting operators admitting a cyclic vector (the image of $1\in \mathbb{C}[x,y]$ in the quotient). So we obtain a map from the punctual Hilbert scheme to the space of commuting matrices with cyclic vector modulo conjugation. This is in fact a bijection:
$$\Hilb^n(\mathbb{C}^2) \cong \{A,B \in \End{\mathbb C}^n, v\in \mathbb{C}^n \mid [A,B]=0, \text{$v$ is a cyclic vector}\} / \text{ conj.}$$
The inverse map associates to an equivalence class $[(A,B)]$ the ideal $I=\{P\in \mathbb{C}[x,y] \mid P(A,B)=0\}$. The existence of the cyclic vector guarantees that $I$ is of codimension $n$.

To end the section, we also introduce the notion of the \textbf{reduced Hilbert scheme}, denoted by $\Hilb^n_{red}(\mathbb{C}^2)$. Geometrically, it corresponds to configurations of $n$ points with barycenter 0. The punctual Hilbert scheme is topologically a direct product of $\mathbb{C}^2$ and the reduced Hilbert scheme. So the dimension of $\Hilb^n_{red}(\mathbb{C}^2)$ is $2n-2$. Its precise definition as a symplectic quotient will be given in Section \ref{section5}. The main property is that the zero-fiber is a Lagrangian subspace of the reduced Hilbert scheme (see Proposition \ref{lagr}).

\section{Higher complex structures}\label{section4}

In this section, we define the higher complex structure and explore its main properties. In order to define a moduli space of higher complex structures, we need to enlarge the group of diffeomorphisms of $\Sigma$ to the space of Hamiltonian diffeomorphisms of the cotangent space $T^*\Sigma$ preserving the zero section. We then explore the local and global theory of that new structure.

\subsection{Definition and basic properties}

In Section \ref{section2}, we saw that a complex structure on a surface $\Sigma$ is uniquely given by a section $\sigma$ of $\mathbb{P}(T^{*\mathbb{C}}\Sigma)$, the (pointwise) projectivized complexified cotangent space, such that at any point $z \in \Sigma$, $\sigma(z)$ and $\bar{\sigma}(z)$ are distinct.
In the previous section, we saw that the projectivization is a special case of the zero-fiber of the punctual Hilbert scheme for $n=2$: $\mathbb{P}(T^{*\mathbb{C}}\Sigma) = \Hilb^2_0(T^{*\mathbb{C}}\Sigma)$. It is now easy to guess the generalization. 

\begin{definition}
A \textbf{higher complex structure} of order $n$ on a surface $\Sigma$, in short \textbf{$n$-complex structure}, is a section $I$ of $\Hilb^n_0(T^{*\mathbb{C}}\Sigma)$ such that at each point $z$ we have that the sum $I(z)+\bar{I}(z)$ is the maximal ideal supported at zero of $T_z^{*\mathbb{C}}\Sigma$.
\end{definition}

\begin{Remark} The space $\Hilb^n(T^{*\mathbb{C}}\Sigma)$ is the pointwise application of $\Hilb^n$ to $T^{*\mathbb{C}}_z\Sigma$ at every point $z$ of $\Sigma$. So it is a bundle of Hilbert schemes. We call it \textbf{Hilbert scheme bundle}. The same holds for the zero-fiber.
\end{Remark}

\noindent For $n=2$, the condition that $I + \bar{I}$ is maximal simply reads $\mu_2\bar{\mu}_2 \neq 1$ which is exactly what we had for the complex structure. So we recover the complex structure.

In the previous section, we saw that not all points in the zero-fiber can be written in the form of a Lagrange interpolation polynomial passing through the origin. Another important consequence of the extra condition is that it rules out non-generic ideals:

\begin{prop}\label{genericideal}
For a $n$-complex structure $I$, we can write at a point $z$ either $I(z)$ or its conjugate $\bar{I}(z)$ as $$\left\langle p^n, -\bar{p}+\mu_2(z,\bar{z})p+...+\mu_n(z,\bar{z})p^{n-1}  \right\rangle  \text{ with } \mu_2\bar{\mu}_2<1.$$ 
\end{prop}

\noindent The proof is a simple computation and can be found in the appendix \ref{appendix0}.

We call the coefficients $\mu_k$ \textbf{higher Beltrami differentials}. The proposition allows us to think of a higher complex structure as a polynomial curve in the cotangent fiber attached to each point of the surface. It is useful to either think of a "hairy" surface (with polynomial curved hairs), or a complex surface germ along the zero-section $\Sigma$ in $T^{*\mathbb{C}}\Sigma$.

In the case where $I$ is of the above form (and not its conjugate $\bar{I}$), we call the $n$-complex structure \textbf{compatible}. This is analogous to the case if Beltrami differential is of norm smaller than 1 or bigger than 1. One can pass from one to the other by conjugation.

Let us compute the global nature of the higher Beltrami differentials. We will see that $\mu_2$ is just the usual Beltrami differential, so of type $(-1,1)$.
Under a holomorphic coordinate transform $z \to z(w)$, we have $p = \frac{\partial}{\partial z} \mapsto \frac{dw}{dz}\frac{\partial}{\partial w}$ and similarly for $\bar{p}$. Hence, 
\begin{align*}
& \left\langle p^n, -\bar{p}+\mu_2(z,\bar{z})p+...+\mu_n(z,\bar{z})p^{n-1}  \right\rangle \\
& \mapsto  \left\langle \left(\frac{dw}{dz}\right)^n \left(\frac{\partial}{\partial w}\right)^n, -\frac{d\bar{w}}{d\bar{z}}\frac{\partial}{\partial \bar{w}}+\frac{dw}{dz}\mu_2(z,\bar{z})\frac{\partial}{\partial w}+...+\left(\frac{dw}{dz}\right)^{n-1}\mu_n(z,\bar{z})\left(\frac{\partial}{\partial w}\right)^{n-1}  \right\rangle \\
& = \left\langle \left(\frac{\partial}{\partial w}\right)^n, -\frac{\partial}{\partial \bar{w}}+\frac{d\bar{z}/d\bar{w}}{dz/dw}\mu_2(z,\bar{z})\frac{\partial}{\partial w}+...+\frac{d\bar{z}/d\bar{w}}{(dz/dw)^{n-1}}\mu_n(z,\bar{z})\left(\frac{\partial}{\partial w}\right)^{n-1}  \right\rangle
\end{align*}

\noindent Thus, we see that for $m=2,...,n$ we get $$\mu_{m}(w,\bar{w}) = \frac{d\bar{z}/d\bar{w}}{(dz/dw)^{m-1}}\mu_{m}(z,\bar{z}).$$
So $\mu_{m}$ is of type $(-m+1,1)$, i.e. a section of $K^{-m+1}\otimes \bar{K}$ (where $K$ is the canonical line bundle).

\subsection{Higher diffeomorphisms}

We wish to define a moduli space of higher complex structures which is finite-dimensional. Higher complex structures "live" in a neighborhood of the zero-section of $T^{*\mathbb{C}}\Sigma$. A diffeomorphism of $\Sigma$ extends linearly to $T^*\Sigma$ and its complexification. Roughly speaking it can only act linearly on the polynomial curve (given by the higher complex structure), so the quotient of all $n$-complex structures by diffeomorphisms of $\Sigma$ is infinite-dimensional. Therefore, we have to enlarge the equivalence relation and quotient by a larger group. What we need are polynomial changes in the cotangent bundle. These can be obtained by symplectomorhpisms of $T^*\Sigma$ generated by a Hamiltonian (a function on $T^*\Sigma$):

\begin{definition}
A \textbf{higher diffeomorphism} of a surface $\Sigma$ is a Hamiltonian diffeomorphism of $T^*\Sigma$ preserving the zero-section $\Sigma \subset T^*\Sigma$ setwise. The set of higher diffeomorphisms is denoted by $\Symp_0(T^*\Sigma)$.

\noindent A \textbf{higher vector field} is a Hamiltonian vector field on the cotangent bundle tangent to the zero-section. The space of such vector fields is denoted by $\Ham_0(T^*\Sigma)$.

\noindent We say that a higher vector field is \textbf{of order $n$} if its Hamiltonian $H$ is a homogenous polynomial of degree $n$.
\end{definition}

\noindent A higher vector field of order 1 is just a vector field of the surface $\Sigma$ linearly extended to the cotangent space.

Preserving the zero-section means that the Hamiltonian $H$ of a higher vector field can be chosen to vanish on the zero-section. In coordinates, this means that one can write the Hamiltonian as $H(z,\bar{z},p,\bar{p}) = \sum_{k,l} w_{k,l}(z,\bar{z})p^k\bar{p}^l$ with $w_{0,0}= 0$. Furthermore, since the Hamiltonian is a real function written in complex coordinates, we have the condition $w_{l,k} = \overline{w_{k,l}}$.

Let us see how a higher diffeomorphism acts on the $n$-complex structure. Roughly speaking, a diffeomorphism of the cotangent bundle acts on the space of sections, and so also on the space of $n$-tuples of sections (corresponds to $n$ points in each fiber), so also on the zero-fiber of the Hilbert scheme (which can be seen as the limit when all $n$ points collapse to the origin). In this way, a higher diffeomorphism acts on a $n$-complex structure.

To be more precise, we need first to understand the variation of an ideal in the space of all ideals:

\begin{prop}
The space of infinitesimal variations of an ideal $I$ in a ring $A$ is the set of all $A$-module homomorhpisms from $I$ to $A/I$.
\end{prop}


\noindent So to compute the variation of an ideal, all we need is to compute the variation of its generators modulo $I$.
These generators are polynomial functions. A general fact of symplectic geometry asserts that the variation of a function $f$ under a Hamiltonian flow generated by a Hamiltonian $H$ is given by the Poisson bracket $\{H,f\}$.
Therefore to compute the action of a higher diffeomorphism generated by a Hamiltonian $H$, we only have to compute the Poisson bracket of $H$ with the generators of $I$, modulo~$I$.

In coordinates, we find that the variation of the generators of the ideal $I$ is given by $\{H,p^n\} \mod I$ and $\{H,-\bar{p}+\mu_2p+...+\mu_np^{n-1}\} \mod I$.

Since we mod out by $I$, a Hamiltonian of order $n$ or higher will have no effect on a $n$-complex structure. So only polynomial Hamiltonians of degree at most $n-1$ act nontrivially. The Hamiltonians of degree $\ge n$ generate a normal subgroup in $\Symp_0(T^*\Sigma)$ acting trivially on $n$-complex structure. Therefore the higher diffeomorphisms act through its quotient. In the following subsection we compute precisely the variation of the higher Beltrami differentials under a Hamiltonian and we deduce the local theory of higher complex structures.

\subsection{Local theory}

In this subsection, we are in an open neighborhood of 0 in $\mathbb{C}$. We will prove that a higher complex structure can locally be trivialized by a higher diffeomorphism, which means that we can make $\mu_2(z,\bar{z})=...=\mu_n(z,\bar{z})=0$ in a neighborhood of 0. Before doing so, we have to compute the variation of the higher complex structure by a higher diffeomorphism.

As seen in the previous subsection, we have to compute Poisson brackets modulo the ideal $I$. A small argument will simplify a lot the computations:

\begin{lemma}
Let $I=\left\langle f_1, ..., f_r \right\rangle$ be an ideal of $\mathbb{C}[z,\bar{z},p,\bar{p}]$ such that $\{f_i, f_j\} = 0 \mod I$ for all $i$ and $j$. Then for all polynomials $H$ and all $k \in \{1,...,r\}$ we have $\{H, f_k\} \mod I = \{H \mod I, f_k\} \mod I$.
\end{lemma}
\begin{proof}
The only thing to show is that if we replace $H$ by $H+gf_l$ for some polynomial $g$ and some $l \in \{1,...,r\}$, the expression does not change. Indeed, 
$\{H+g f_l, f_k\}=\{H, f_k\}+g\{f_l,f_k\}+\{g,f_k\}f_l = \{H, f_k\}  \mod I$ using the assumption.
\end{proof}

\noindent For our ideal $I = \left\langle p^n,-\bar{p}+\mu_2p+...+\mu_np^{n-1} \right \rangle$, we have $\{p^n,-\bar{p}+\mu_2 p+...+\mu_n p^{n-1}\} = np^{n-1}(\partial\mu_2 p+...+\partial\mu_n p^{n-1})=0 \mod I$, so we can use the proposition. A Hamiltonian $H$ modulo $I$ can thus always be written as $H=\sum_k w_k p^k$.

Another small argument is that the first generator, $p^n$, does not change : $\{H,p^n\} \mod I = np^{n-1}\partial H \mod I = 0$ since there is no constant term in $H$. So we only have to settle the second generator.

Let us compute the variation of the second generator under the Hamiltonian $H=w_kp^k$.
\begin{align*}
\{w_kp^k, -\bar{p}+\mu_2p+...+\mu_np^{n-1}\} =& kw_kp^{k-1}(\partial\mu_2p+...+\partial\mu_np^{n-1}) \\
& -\partial w_kp^k(\mu_2+2\mu_3p+...+(n-1)\mu_np^{n-2})+\bar{\partial}w_k p^k \\
=& p^k(\bar{\partial}w_k-\mu_2\partial w_k+k\partial\mu_2 w_k) \\
&+\sum_{l=1}^{n-1-k} p^{k+l} (k\partial\mu_{l+2}w_k-(l+1)\mu_{l+2}\partial w_k) \mod I
\end{align*}

\noindent Thus we obtain
\begin{prop}\label{varmu}
The variations $\delta \mu_l$ are given by 
$$\delta \mu_l = \left \{ \begin{array}{cl}
(\bar{\partial}-\mu_2\partial+k\partial\mu_2)w_k & \text{ if  } l=k+1 \\
(k\partial\mu_{l-k+1}-(l-k)\mu_{l-k+1}\partial)w_k &\text{ if  } l>k+1 \\
0 & \text{ if  } l<k+1.
\end{array} \right.$$
\end{prop}

\noindent Now, we are ready to state the local triviality of higher complex structure:

\begin{thm}\label{loctrivial}
The $n$-complex structure can be locally trivialized, i.e. there is a higher diffeomorphism which sends the structure to $(\mu_2(z,\bar{z}),...,\mu_n(z,\bar{z}))=(0,...,0)$ for all small $z$.
\end{thm}

\noindent The proof is in the spirit of the classical proof of Darboux theorem on local theory of symplectic structures. This can be found in the appendix \ref{appendix1}.


\subsection{Moduli space}

We are finally ready to define and study the moduli space of higher complex structures. We then show that it is a contractible ball of dimension $(n^2-1)(g-1)$ and describe its tangent and cotangent space.

\begin{definition}
The space of all compatible $n$-complex structures modulo higher diffeomorphisms is called the \textbf{geometric Hitchin space} and denoted by $\hat{\mathcal{T}}^n_{\Sigma}$. 
In formula: $$\hat{\mathcal{T}}^n_{\Sigma} = \Gamma(\Hilb^n_0(T^{*\mathbb{C}}\Sigma)) / \Symp_0(T^*\Sigma)$$
\end{definition}

\noindent Recall that an $n$-complex structure is compatible with orientation if the ideals $I$ are of the form $\left \langle p^n, -\bar{p}+\mu_2 p+...+\mu_n p^{n-1} \right \rangle$ with $\mu_2\bar{\mu}_2<1$. Using complex conjugation we get another moduli space for which $\bar{I}$ is of this form.

Since a higher diffeomorphism of order 1 is a usual diffeomorphism and only Hamiltonians of order at most $n-1$ act non-trivially on $n$-complex structures, we recover for $n=2$ the usual Teichm\"uller space: $$\hat{\mathcal{T}}^2_{\Sigma}=\mathcal{T}_{\Sigma}.$$
Our main result is
\begin{thm}
For a surface $\Sigma$ of genus $g\geq 2$ the geometric Hitchin space $\hat{\mathcal{T}}^n_{\Sigma}$ is a contractible manifold of complex dimension $(n^2-1)(g-1)$. In addition, its cotangent space at any point $\mu=(\mu_2,...,\mu_n)$ is given by 
$$T^*_{\mu}\hat{\mathcal{T}}^n_{\Sigma} = \bigoplus_{m=2}^{n} H^0(\Sigma,K^m)$$
In addition, there is a forgetful map $\hat{\mathcal{T}}^n_{\Sigma} \rightarrow \hat{\mathcal{T}}^{n-1}_{\Sigma}$.
\end{thm}

\noindent We essentially show the existence of a cotangent space at every point, so we have a manifold. The dimension of the cotangent space is computed by Riemann-Roch. We do not enter into details on issues about infinite-dimensional manifolds and quotients. 

\begin{Remark}
The proof for the contractibility in the initial publication (IMRN 20/2021) is erroneous. Many thanks to Alex Nolte for pointing this out to us. A correct proof using the notion of harmonic $n$-complex structures can be found in his paper \cite{Nolte}. 
\end{Remark}

\begin{Remark}
As pointed out by Alex Nolte to us, there is a lack in the proof below: the existence of a tangent space at every point is not sufficient to be a manifold. The missing piece is the property of being Hausdorff. The action of higher diffeomorphisms being non-proper, this issue is non-trivial. We refer to his paper \cite{Nolte} for a proof that the moduli space $\hat{\mathcal{T}}^n_{\Sigma}$ is indeed a Hausdorff topological space.
\end{Remark}

\begin{proof}
We start by showing the manifold structure. For that, we examine the infinitesimal variation around any point. This will also give a description of the tangent and cotangent space. By definition, we have $$\hat{\mathcal{T}}^n_{\Sigma} = \{(\mu_2,...,\mu_n) \mid \mu_m \in  K^{-m+1}\otimes \bar{K} \; \forall \, m \text{ and } \left|\mu_2\right| < 1\} / \Symp_0(T^*\Sigma)$$
The infinitesimal variation around $\mu=(\mu_2,...,\mu_n)$ is then given by 
$$T_\mu\hat{\mathcal{T}}^n_{\Sigma} = \{(\delta\mu_2,...,\delta\mu_n) \mid \delta\mu_m \in  K^{-m+1}\otimes \bar{K} \;\forall \, m \} / \Ham_0(T^*\Sigma).$$

\noindent In the previous subsection, we have seen that every $n$-complex structure is locally trivializable. So there is an atlas in which $\mu:=(\mu_2,...,\mu_n)=0$.
In addition we have computed the action of a Hamiltonian vector field on the $n$-complex structure in proposition \ref{varmu}. Locally, we can decompose the Hamiltonian into homogeneous parts of degrees 1 to $n-1$. All higher terms do not affect the $n$-complex structure. By proposition \ref{varmu} (with $\mu_k=0$ for all $k$), we get 
$$T_\mu\hat{\mathcal{T}}^n_{\Sigma} = \{(\delta\mu_2,...,\delta\mu_n) \} / (\bar{\partial}w_1,...,\bar{\partial}w_{n-1})$$ where $w_m$ is a section of $K^{m}$. Thus, the tangent space splits into parts $$T_\mu\hat{\mathcal{T}}^n_{\Sigma} = \{\delta\mu_2 \in \bar{K}\otimes K^{-1}\} / \bar{\partial}w_1 \oplus ... \oplus \{\delta\mu_n \in K^{-n+1}\otimes \bar{K}\} / \bar{\partial}w_{n-1}$$
To compute the cotangent space, we use the pairing between differential of type $(1-k,1)$ and of type $(k,0)$ given by integration over the surface. We get 
\begin{align*}
(\{\delta\mu_m\} / \bar{\partial}w_{m-1})^* =& \; \{t_m \in K^m \mid \smallint t_m \bar{\partial}w_{m-1} = 0 \; \forall \, w_{m-1} \in K^{m-1} \} \\
=& \; \{t_m \in K^m \mid \smallint \bar{\partial}t_m w_{m-1} = 0 \; \forall \, w_{m-1} \in K^{m-1} \} \\
=& \; \{t_m \in K^m \mid \bar{\partial}t_m = 0 \} \\
=& \; H^0(\Sigma,K^m)
\end{align*}

\noindent Therefore $$T^*_\mu\hat{\mathcal{T}}^n_{\Sigma} = \bigoplus_{m=2}^{n}H^0(\Sigma,K^m).$$

\noindent Now, a standard computation using Riemann-Roch formula shows that (using genus $g\geq 2$)
$$\dim H^0(\Sigma,K^m) = (2m-1)(g-1).$$
Therefore $$\dim{\hat{\mathcal{T}}^n_{\Sigma} } = \dim{T^*_\mu\hat{\mathcal{T}}^n_{\Sigma} } = \sum_{m=2}^{n} \dim{H^0(K^m)} = \sum_{m=2}^{n} (2m-1)(g-1) = (n^2-1)(g-1).$$

\noindent The forgetful map is simply given by the following: to the equivalence class of a $n$-complex structure given by an ideal $I$, we associate the equivalence class of $I + \left\langle p,\bar{p} \right\rangle^{n-1}$. This is independent of coordinates since $\left\langle p,\bar{p} \right\rangle$ is the maximal ideal supported on the origin. In coordinates, the map just forgets $\mu_n$.
\end{proof}

\noindent Composing the forgetful maps, we get a map from $\hat{\mathcal{T}}^n_{\Sigma}$ to Teichm\"uller space. To any $n$-complex structure is therefore associated a complex structure. The cotangent space of geometric Hitchin space at $I$ is the Hitchin base of holomorphic differentials where holomorphicity is with respect to the associated complex structure of $I$.

Since the cotangent space is a complex vector space, we have automatically an almost complex structure on the geometric Hitchin space. Furthermore, since it is the moduli space of a geometric structure on the surface, the mapping class group acts properly discontinuously on it.

Let us compare geometric Hitchin space to Hitchin's component: both are contractible real manifolds of the same dimension. Hitchin's component has a symplectic structure (Goldman's symplectic structure on character varieties), but no obvious complex structure, nor forgetful maps. Geometric Hitchin space has (for the moment) no obvious symplectic structure, but enjoys all the properties we explained before. An isomorphism between both would enrich both sides.

\begin{conj}
Geometric Hitchin space $\hat{\mathcal{T}}^n_{\Sigma}$ is canonically isomorphic to Hitchin's component $\mathcal{T}^n_{\Sigma}$.
\end{conj}

\noindent One way to attack the conjecture is the following idea: To any $n$-complex structure, we can associate canonically a vector bundle of rank $n$ whose fiber over a point $z$ is $\mathbb{C}[p,\bar{p}]/I(z)$. The task is then to find a flat connection in this bundle with monodromy in $PSL_n(\mathbb{R})$. This amounts to a map from $\pi_1(\Sigma)$ to $PSL_n(\mathbb{R})$, so a point of the character variety.

In the last section we give an extra argument in favor of this conjecture: as for Hitchin's component, we get a spectral curve.

\section{A spectral curve}\label{section5}

In this final section, we explore and exploit the symplectic structure of the Hilbert scheme in order to construct a spectral curve associated to a cotangent vector to higher Hitchin space.

\subsection{Symplectic structure of punctual Hilbert scheme}
The punctual Hilbert scheme inherits a complex symplectic structure from the space of $n$ points $(\mathbb{C}^2)^n$. Denoting by $(x_i,y_i)$ the coordinates of the $i$-th point, the symplectic structure is simply $\omega=\sum_i dx_i\wedge dy_i$. We now show that this expression extents to the blow up (which gives the Hilbert scheme). We simply express $\omega$ in terms of coordinates $t$ and $\mu$.
For this, denote by $M_x$ and $M_y$ the multiplication operators by $x$ and $y$ respectively in $\mathbb{C}[x,y]/I$ where $I$ is a generic element of $\Hilb^n(\mathbb{C}^2)$: $$I=\left\langle x^n+t_1x^{n-1}+...+t_n, -y+\mu_1+\mu_2x+...+\mu_nx^{n-1} \right\rangle$$
Since $x^n+t_1x^{n-1}+...+t_n = \prod_i (x-x_i)$, we see that $M_x$ can be diagonalized to $\diag(x_1,...,x_n)$. Since $y=Q(x)$, we get $M_y=Q(M_x)$. So its diagonalized form is $\diag(Q(x_1),...,Q(x_n)) = \diag(y_1,...,y_n)$ since $Q$ is the interpolation polynomial. Since the trace is unchanged by conjugation, we get
$$\omega = \tr \diag(dx_1,...,dx_n)\wedge\diag(dy_1,...,dy_n) = \tr dM_x\wedge dM_y$$
In the basis $(1,x,x^2,...,x^{n-1})$, $M_x$ is a companion matrix, so $dM_x$ has only non-zero elements in the last column. Denote by $\alpha_{i,j}$ the matrix elements of $M_y$ in this basis. These can be expressed in terms of $\mu$ and $t$. Then we have $$\omega = \sum_i dt_i\wedge d\alpha_{n,n+1-i}.$$ 
This gives the symplectic structure on $\Hilb^n(\mathbb{C}^2)$.

\begin{Remark}
The symplectic structure on the Hilbert scheme can also be described in terms of the Poisson bracket between the coordinates. We have $\{t_i,t_j\} = 0 = \{\mu_i,\mu_j\}$ and $\{\mu_i,t_j\} = t_{j-i}$ with $t_0=1$ and $t_i=0$ for $i<0$.
\end{Remark}

Now we come back to the reduced Hilbert scheme and show that the zero-fiber $\Hilb^n_0(\mathbb{C}^2)$ is a Lagrangian subspace.
Consider the action of $\mathbb{C}$ on $\mathbb{C}^2$ given by translation in the $y$ direction. This action induces the action on $(\mathbb{C}^2)^n$, which action is Hamiltonian with moment map $H(x,y)=x_1+\cdots+x_n$. Let us do the Hamiltonian reduction (Marsden-Weinstein quotient): first we restrict to $H^{-1}(\{0\})$ which corresponds to $t_1=0$ since by Vieta's correspondence $t_1=-x_1+\cdots-x_n=-H(x,y)$. Then we have to quotient out the action which means that we have to identify $\{y_i\}$ with $\{y_i+t\}$ for all $i$, i.e. modulo the vertical shift in $\mathbb{C}^2$. Geometrically, we want the barycenter of the $n$ given points being zero. With this $\mu_1$ becomes a function of the other $\mu$'s and the $t$'s.

\begin{definition}
We define the reduced punctual Hilbert scheme of the plane, denoted by $\Hilb^n_{red}(\mathbb{C}^2)$, to be the Hamiltonian reduction $\Hilb^n(\mathbb{C}^2)//_0 \mathbb{C}$.
\end{definition}

\noindent Since the reduced Hilbert scheme is a Hamiltonian reduction of a symplectic space, it is itself symplectic. Its symplectic structure is simply the one from $\Hilb^n(\mathbb{C}^2)$ with $t_1=0$ ($\mu_1$ only appears together with $t_1$ so also disappears).

\begin{prop}\label{lagr}
The zero-fiber $\Hilb^n_0(\mathbb{C}^2)$ is Lagrangian in $\Hilb^n_{red}(\mathbb{C}^2)$.
\end{prop}
\begin{proof}
Since for the zero-fiber, all $t_i=0$, we see that $\omega= \sum_i dt_i\wedge d\alpha_{n,n+1-i}$ vanishes on it. Since its dimension is $n-1$ which is half of the dimension of the reduced Hilbert scheme, we are done.
\end{proof}

\noindent As a corollary, we get that the cotangent space to the zero-fiber $\Hilb^n_0(\mathbb{C}^2)$ is given by its normal bundle inside the reduced Hilbert scheme. At first order, i.e. modulo $t^2$ (terms which are at least quadratic in the $t_i$'s), this is the whole space $\Hilb^n_{red}(\mathbb{C}^2)$.

\subsection{Cotangent space to geometric Hitchin space}
We already described the cotangent space at one point $I$ to geometric Hitchin space. Here we describe the total cotangent space $T^*\hat{\mathcal{T}}^n$.

Analyzing the global nature of the matrix elements $\alpha_{i,j}$ of $M_y$, one can easily show that $t_i(z)\alpha_{n,n+1-i}(z)$ is of type (1,1), so can be readily integrated over $\Sigma$. Thus, the symplectic structure of the Hilbert scheme extends naturally to a symplectic structure of the Hilbert scheme bundle $\Hilb^n(T^{*\mathbb{C}}\Sigma)$ given by $$\omega = \int_{\Sigma} \sum_i dt_i(z)\wedge d\alpha_{n,n+1-i}(z).$$

\noindent By the last paragraph of the previous subsection, we see that near the zero-section, the total cotangent space $T^*\hat{\mathcal{T}}^n$ is nothing but a section of the reduced Hilbert scheme bundle, so an ideal of the form $$\left\langle p^n+t_2(z)p^{n-2}+...+t_n(z),-\bar{p}+\mu_1(z)+\mu_2(z)p+...+\mu_n(z)p^{n-1} \right\rangle.$$
Of course, we have to quotient out the action of the higher diffeomorphisms.

We already computed the variation of $\mu_k$ under a Hamiltonian in proposition \ref{varmu}. Imitating the same computation as for the cotangent space at one point, but now around an arbitrary $n$-complex structure $(\mu_2,...,\mu_n)$, we get the following:

\begin{thm}
The cotangent bundle of the geometric Hitchin space is given by 
\begin{align*}
T^*\hat{\mathcal{T}}^n=\{& (\mu_2, ..., \mu_n, t_2,...,t_n) \mid  \mu_k \in \Gamma(K^{1-k}\otimes \bar{K}), t_k \in \Gamma(K^k) \text{ and } \; \forall k\\
& (-\bar{\partial}\!+\!\mu_2\partial\!+\!k\partial\mu_2)t_{k}+\sum_{l=1}^{n-k}((l\!+\!k)\partial\mu_{l+2}+(l\!+\!1)\mu_{l+2}\partial)t_{k+l}=0 \}
\end{align*}
\end{thm}
\noindent Notice that for $\mu=0$, we get $\bar{\partial}t_k=0$, i.e. $t_k \in H^0(K^k)$ a holomorphic differential as previously computed. Some details can be found in the appendix \ref{appendix2}.

\subsection{Spectral curve}
In this subsection, we construct a ramified cover $\tilde{\Sigma}$ over the surface $\Sigma$ inside its complexified cotangent bundle $T^{*\mathbb{C}}\Sigma$ associated to a point of the cotangent bundle of the $n$-complex structures $T^*\hat{\mathcal{T}}^n$.

Define polynomials $P(p)=p^n+t_2p^{n-2}+...+t_n$ and $Q(p, \bar{p})=-\bar{p}+\mu_1+\mu_2 p+...+\mu_n p^{n-1}$ where $\mu_1$ is an explicit function of the other variables given by the reduced Hilbert scheme. Put $I= \left\langle P(p), Q(p, \bar{p}) \right\rangle$. Define $\tilde{\Sigma} \subset T^{*\mathbb{C}}\Sigma$ by the zero set of $P$ and $Q$. This curve $\tilde{\Sigma}$ is called \textbf{spectral curve}. This is a ramified covering space with $n$ sheets.

\begin{prop}\label{spectralcurve}
We have $\{ P,Q\} = 0 \mod I \mod t^2$ iff $I \in T^*\hat{\mathcal{T}}^n$.
\end{prop}
\noindent The proposition means that the spectral curve is Lagrangian "near the zero-section" iff the ideal comes from the cotangent space of the $n$-complex structures. Concretely this means that the $t_k$ satisfy the condition  appearing in the description of the total cotangent space $T^*\hat{\mathcal{T}}^n$.

The proof is a direct computation. Expanding the result in $p$ gives the conditions of the previous theorem as coefficients. The coefficient of $p^{n-1}$ vanishes because of the special value of $\mu_1$. For details see the appendix \ref{appendix3}.

To a cotangent vector of a $n$-complex structure is associated a spectral curve. Notice that the spectral curve here is independent of a complex structure on the surface $\Sigma$, but it lies in the complexified cotangent bundle. Hitchin's spectral curve depends on the complex structure and lies in the holomorphic cotangent bundle (so it is trivially Lagrangian as a one-dimensional subspace of a two-dimensional symplectic space). For $\mu_k=0$ for all $k$, our spectral curve $\tilde{\Sigma}$ lies in the holomorphic cotangent bundle and can be identified with Hitchin's spectral curve (with complex structure $\mu_2=0$ on $\Sigma$). So $\tilde{\Sigma}$ can be seen as Hitchin's spectral curve deformed into the $\bar{p}$-direction of $T^{*\mathbb{C}}\Sigma$ by the $n$-complex structure.

On the spectral curve $\tilde{\Sigma}$, there is a line bundle $L$ with fiber the eigenspace of $M_p$, the multiplication operator by $p$ in $\mathbb{C}[p,\bar{p}]/I$, since the characteristic polynomial of $M_p$ is given by $P$. The pushforward of $L$ to $\Sigma$ by the covering map gives the rank $n$ bundle with fiber $\mathbb{C}[p,\bar{p}]/I$. So we get a similar spectral data as for Hitchin's spectral curve.

Since $\tilde{\Sigma}$ is Lagrangian to order 1, we can compute its periods. The ratios of periods at the limit when all $t$'s go to 0 (so $\tilde{\Sigma}$ collapses to $\Sigma$) should give coordinates on geometric Hitchin space $\hat{\mathcal{T}}^n\Sigma$. 


\section{Appendix}
\label{section6}

In this appendix, we give details for proofs and computations.

\subsection{Proof of proposition \ref{genericideal}} \label{appendix0}

\begin{fakeprop}
For a $n$-complex structure $I$, we can write at a point $z$ either $I(z)$ or its conjugate $\bar{I}(z)$ as $$\left\langle p^n, -\bar{p}+\mu_2(z,\bar{z})p+...+\mu_n(z,\bar{z})p^{n-1}  \right\rangle  \text{ with } \mu_2\bar{\mu}_2<1.$$ 
\end{fakeprop}

\begin{proof}
The proposition concerns a cotangent fiber of one point $z$. So we can really work on $\mathbb{C}^2$ with coordinates $(p,\bar{p})$. 
Let $I_1$ be the set of all degree 1 polynomials which appear as elements of $I$. It is clear that $I_1$ is a vector subspace of $\mathbb{C}^2$ since $I$ is a vector space. We will show that $I_1$ is of dimension 1. 

\noindent If $I_1=\{0\}$, then so is $\bar{I}_1=\{0\}$. But by $I\oplus \bar{I} = \left\langle p,\bar{p} \right\rangle$, we get $I_1\oplus\bar{I}_1 = \mathbb{C}^2$ which is absurd. 
If $I_1=\mathbb{C}^2$ then $I=\left\langle p,\bar{p}\right\rangle$ which contradicts the fact that it is of codimension $n\geq 2$. 

\noindent Therefore $I_1=\Vect(ap + b\bar{p})$ is of dimension 1. So $\bar{I}_1=\Vect(\bar{a} \bar{p} + \bar{b} p)$ and the condition $I\oplus \bar{I} = \left\langle p,\bar{p} \right\rangle$ is equivalent to 
$a\bar{a}\neq b\bar{b}$. Assume $a\bar{a}< b\bar{b}$ (the other case being similar and leads to $\bar{I}$ instead of $I$), then $I_1=\Vect(-\bar{p}+\mu_2p)$ with $\left|\mu_2\right|=\left|a/b\right|<1$.

\noindent Finally, since $-\bar{p}+\mu_2p \in I_1$, there is a relation of the form $\bar{p}=\mu_2p+\text{higher terms}$ in $I$. Iterating this equality by replacing it in any $\bar{p}$ appearing in the higher terms, we will get an expression of $\bar{p}$ in terms of monomials in $p$. 
Since $p^n=0$ in $I$, we get $$\bar{p}=\mu_2p+\mu_3p^2+...+\mu_np^{n-1} \mod I.$$
To give an example, we get for $n=4$ and $\bar{p}=ap+bp\bar{p}$ that $$\bar{p}=ap+bp(ap+b(ap+b\bar{p}))=ap+abp^2+ab^2p^3.$$ \end{proof}

\subsection{Local Triviality}\label{appendix1}
Here we give the proof of theorem \ref{loctrivial}:
\begin{fakethm}
The $n$-complex structure can be locally trivialized, i.e. there is a higher diffeomorphism which sends the structure to $(\mu_2(z,\bar{z}),...,\mu_n(z,\bar{z}))=(0,...,0)$ for all small $z$.
\end{fakethm}

\begin{proof}
The proof is by induction. For $n=2$, we already know the result which is Gauss and Korn-Lichtenstein's theorem on the existence of isothermal coordinates. So suppose that the statement is true for $n \geq 2$ and we will show it for $n+1$.

By induction hypothesis, there is a higher diffeomorphism which makes $\mu_2(z)=...=\mu_n(z)=0$ for all $z$ near the origin. We will construct a higher diffeomorphism generated by a Hamiltonian of degree $n$ giving $\mu_{n+1}(z)=0$ for all $z$ near 0. Since a Hamiltonian of degree $n$ does not affect the $\mu_k$ with $k\leq n$ (see previous proposition), we are done.

Let us try a Hamiltonian of the form $$H(z, \bar{z},p,\bar{p})=w_n(z,\bar{z},p,\bar{p})p^n+\bar{w}_n(z,\bar{z},p,\bar{p})\bar{p}^n$$ generating a flow $\phi_t$. 
We denote by $\mu_{n+1}^t(z,\bar{z})$ the image of $\mu_{n+1}(z,\bar{z})$ by $\phi_t$ (note that $\phi_t$ fixes the zero-section pointwise). The variation formula  for $\mu_2=0$ then reads
$$\frac{d}{dt}\mu_{n+1}^{t}(z,\bar{z}) = \bar{\partial}w_{n}(z,\bar{z},0,0)$$
Thus, the variation does not depend on time. We wish to have $\frac{d}{dt}\mu_{n+1}^{t}(z,\bar{z})=-\mu_{n+1}^{t=0}(z,\bar{z}).$ So we have to solve 
$$\bar{\partial}w_{n}(z,\bar{z},0,0)=-\mu_{n+1}^{0}(z,\bar{z})$$
The inversion of the Cauchy-Riemann operator $\bar{\partial}$ is well-known. We denote its inverse by $T$. Explicitly, we have 
$$Tf(z)=\frac{1}{2\pi i}\int_{\mathbb{C}}\frac{f(\zeta)}{\zeta-z}d\zeta \wedge d\bar{\zeta}$$ for any square-integrable function $f$.

Therefore, on the zero-section we set $w_n(z,\bar{z},0,0)=-T\mu_{n+1}^{0}(z,\bar{z})$ (since $\mu_{n+1}$ is smooth, it is locally square-integrable). To define it everywhere, we choose a bump function $\beta$, in our case a function on $\mathbb{C}^2$ which is 1 in a neighborhood of the origin and 0 outside a bigger neighborhood of the origin, and we put $$w_n(z,\bar{z},p,\bar{p})=-\beta(p,\bar{p}) T\mu_{n+1}^{0}(z,\bar{z}).$$ So the Hamiltonian is defined everywhere and gives a compactly supported vector field which therefore can be integrated for all times. We then get $$\mu_{n+1}^t(z,\bar{z})=(1-t)\mu_{n+1}^{0}(z,\bar{z})$$
Therefore, at time $t=1$, $\mu_{n+1}$ vanishes everywhere.
\end{proof}

\subsection{Cotangent space}\label{appendix2}

As for the case $\mu=0$ we have $$T^*_\mu\hat{\mathcal{T}}^n_{\Sigma}=\{t_2,...,t_n \mid t_i\in \Gamma(K^i) \text{ and } \int_{\Sigma}\sum_k t_k \delta \mu_k =0 \}.$$
We compute using the formula for $\delta\mu_k$ given in proposition \ref{varmu}:
\begin{align*}
\int \sum_k t_k \delta \mu_k =& \int (\bar{\partial}\!-\!\mu_2 \partial\!+\!(k\! -\!1)\partial\mu_2)w \; t_{k}+\sum_{l=1}^{n-k}(-(l\!+\!1)\mu_{l+2}\partial+(k\!-\!1)\partial \mu_{l+2})w \; t_{k+l} \\
=& \int w \left((-\bar{\partial}\!+\!\mu_2\partial\!+\!k\partial\mu_2)t_{k}+\sum_{l=1}^{n-k}((l\!+\!k)\partial\mu_{l+2}+(l\!+\!1)\mu_{l+2}\partial)t_{k+l}  \right)  
\end{align*}

\noindent This has to be zero for all $w$, therefore $$T^*_\mu\hat{\mathcal{T}}^n_{\Sigma}=\{(t_2,...,t_n) \mid  (-\bar{\partial}+\mu_2\partial+k\partial\mu_2)t_{k}+\sum_{l=1}^{n-k}((l+k)\partial\mu_{l+2}+(l+1)\mu_{l+2}\partial)t_{k+l}=0 \; \forall k\}$$

\subsection{Spectral Curve}\label{appendix3}

Here we give some details to proposition \ref{spectralcurve} stating that $\{ P,Q\} = 0 \mod I \mod t^2$ iff $I \in T^*\hat{\mathcal{T}}^n$.

We can decompose the Poisson bracket $\{p^n+t_2p^{n-2}+...+t_n,-\bar{p}+\mu_1+\mu_2p+...+\mu_np^{n-1}\}$ into three parts: the first part $\{t_2p^{n-2}+...+t_n,-\bar{p}+\mu_2p+...+\mu_np^{n-1}\}$ has been already computed by interpreting $t_2p^{n-2}+...+t_n$ as a Hamiltonian $H$. This gives almost the condition of the cotangent space to geometric Hitchin space. The only problem is the term $p^{n-1}$.

The second term $\{t_2p^{n-2}+...+t_n,\mu_1\}$ simply vanishes modulo $t^2$ since $\mu_1 = 0 \mod t$. 

The third and last term $\{p^n, -\bar{p}+\mu_1+\mu_2p+...+\mu_np^{n-1}\} = np^{n-1}(\partial\mu_1+\partial\mu_2p+...+\partial\mu_np^{n-1})$ and the special value of $\mu_1$ let precisely vanish the term with $p^{n-1}.$

\end{document}